\documentclass[10pt,twoside]{article}

\usepackage{xcolor}

\usepackage{lmodern}
\usepackage[T1]{fontenc}
\usepackage[utf8]{inputenc}
\usepackage{textcomp}
\usepackage[english]{babel}

\newcommand{\fulltitle}{Stochastic collocation method for computing eigenspaces of parameter-dependent operators}
\newcommand{\shorttitle}{Eigenspaces of parameter-dependent operators}

\usepackage[top=31mm, bottom=31mm, left=33mm, right=33mm]{geometry}

\usepackage{fancyhdr}
\usepackage{graphicx}
\usepackage{hyperref}
\usepackage{amssymb}
\usepackage{amsthm}
\usepackage{amsmath}
\usepackage{verbatim}
\usepackage{subfig}
\usepackage{enumitem}

\def\a{\alpha}
\def\b{\beta}
\def\g{\gamma}

\def\.{\cdot}

\def\iu{\mathrm{i}}

\newcommand{\R}{\mathbb{R}}
\newcommand{\C}{\mathbb{C}}
\newcommand{\N}{\mathbb{N}}

\newcommand{\norm}[2]{\left\lVert#1\right\rVert_{#2}}

\newcommand{\mbf}[1]{\mathbf{#1}}
\newcommand{\mbb}[1]{\mathbb{#1}}
\newcommand{\mc}[1]{\mathcal{#1}}

\newcommand{\supp}{\mathop{\mathrm{supp}}}

\DeclareMathOperator*{\essinf}{ess\,inf}
\DeclareMathOperator*{\dist}{dist}
\DeclareMathOperator*{\spa}{span}

\newtheorem{thm}{Theorem}
\newtheorem{lmm}{Lemma}
\newtheorem{cor}{Corollary}
\newtheorem{prp}{Proposition}
\newtheorem{rmk}{Remark}

\newtheorem{exm}{Example}

\title{}

\begin{document}

\title{\fulltitle}
\author{Luka Grubi{\v s}i{\' c} \footnote{University of Zagreb, Faculty of Science, Department of Mathematics. Bijeni{\v c}ka 30, 10000 Zagreb, Croatia; E-mail: luka.grubisic@math.hr; The work of the author has been supported by the Croatian Science Foundation grant IP-2019-04-6268. We gratefully acknowledge the support.} \and Harri Hakula \footnote{Aalto University, Department of Mathematics and Systems Analysis, P.O. Box 11100, FI-00076 Aalto, Finland; E-mail: harri.hakula@aalto.fi; The work of the author was supported by the (FP7/2007--2013) ERC grant agreement no 339380.} \and Mikael Laaksonen\footnote{Aalto University, Department of Mathematics and Systems Analysis, P.O. Box 11100, FI-00076 Aalto, Finland; E-mail: mikael.j.laaksonen@aalto.fi; The work of the author was supported by the Magnus Ehrnrooth Foundation.}}

\maketitle

\begin{abstract}
We consider computing eigenspaces of an elliptic self-adjoint operator depending on a countable number of parameters in an affine fashion. The eigenspaces of interest are assumed to be isolated in the sense that the corresponding eigenvalues are separated from the rest of the spectrum for all values of the parameters. We show that such eigenspaces can in fact be extended to complex-analytic functions of the parameters and quantify this analytic dependence in way that leads to convergence of sparse polynomial approximations. A stochastic collocation method on an anisoptropic sparse grid in the parameter domain is proposed for computing a basis for the eigenspace of interest. The convergence of this method is verified in a series of numerical examples based on the eigenvalue problem of a stochastic diffusion operator.

\bigskip


\end{abstract}

\section{Introduction}

Multiparametric eigenvalue problems, i.e., eigenvalue problems of operators that depend on a large number of input parameters, arise in a variety of contexts. One may think of optimization of the spectrum of structures which depend on a number of design parameters, but also uncertainty quantification of engineering systems with data uncertainty. Recent literature has considered examples of mechanical vibration problems, where a parametrization of the uncertainties in either the physical coefficients or the geometry of the system results in a multiparametric eigenvalue problem, see e.g. \cite{verhooselgutierrezhulshoff06}, \cite{meidanighanem14}, \cite{hakulalaaksonen15}, \cite{sousedikelman16}, \cite{hakulakaarniojalaaksonen18}, \cite{hakulalaaksonen19a}.

It is to be noted that multiparametric eigenvalue problems present some additional difficulties when compared to corresponding source problems. First of all, the eigenvalue problem introduces a product of unknowns and hence non-linearities arise. Second of all, one needs to pay special special attention to the selection of the right eigenmodes. For these reasons, techniques developed for the analysis and the numerical solution of the source problems are, in general, not directly applicable in the context of the eigenvalue problems. Nevertheless, in recent years several numerical methods have been suggested for solving multiparametric eigenvalue problems. The focus has been on spectral methods, which are based on polynomial approximations of the solution in the parameter domain and which have been shown to exhibit superior convergence rates compared to traditional Monte Carlo methods \cite{xiuhesthaven05}, \cite{schwabtodor07}, \cite{bieriandreevschwab09}, \cite{cohendevoreschwab10}. These typically take the form of stochastic collocation methods or the form of matrix iterations which rely on stochastic Galerkin approximation of the solution. A benchmark for the first class of methods is the sparse anisotropic collocation algorithm analyzed by Andreev and Schwab in \cite{andreevschwab12}. In the latter class of methods many different variants have been proposed over the years \cite{verhooselgutierrezhulshoff06}, \cite{meidanighanem14}, \cite{hakulakaarniojalaaksonen15}, \cite{hakulalaaksonen19b}. Quite recently, low-rank methods have also been introduced \cite{sirkovic16}, \cite{benneronwuntastoll18}, \cite{elmansu19}.

By their very nature, the spectral methods considered above rely on the assumption that the solution is smooth with respect to the input parameters. More precisely, these methods exhibit optimal rates of convergence only if the eigenpair of interest depends complex-analytically on the vector of parameters. This analytic dependence has been established for nondegenerate eigenvalues and associated eigenvectors in \cite{andreevschwab12}. For such eigenpairs we therefore have optimal rates of convergence for stochastic collocation algorithms, see \cite{andreevschwab12} for details, and optimal asymptotic rates of convergence for the iterative Galerkin based algorithms considered in \cite{hakulalaaksonen19b}. However, these results do not apply to cases where the eigenvalues are of higher multiplicity or where they are allowed to \textit{cross} within the parameter space. As noted in e.g. \cite{hakulalaaksonen19a}, many interesting engineering applications admit eigenvalues that are clustered close together and therefore the aforementioned eigenvalue crossings may not be avoided when these problems are cast into the parameter-dependent setting. 

In some special cases it is possible to identify the eigenmodes by some characteristic features such as Fourier indices. Then it may be possible to 
track eigenpairs through the parameter space by searching for the modes with the given indices even though the ordering of such modes becomes mixed over the parameter space. In other words, one well-defined basis for a given subspace is readily available. An example of such a problem can be constructed by extending the Dirichlet Laplacian on the unit square (Example~\ref{ex:model}).

\begin{exm}[Model problem]\label{ex:model}
Let us consider the Dirichlet Laplacian eigenproblem on the unit square. Note that this can also be seen as an example of the diffusion eigenproblem with a constant diffusion coefficient. It is well-known that the first four eigenpairs are 
\begin{align*}
(\lambda_1,u_1)&=(\lambda_{(1,1)},u_{(1,1)})=(2 \pi ^2, \sin (\pi  x_1) \sin (\pi  x_2)),\\
(\lambda_2,u_2)&=(\lambda_{(2,1)},u_{(2,1)})=(5 \pi ^2,\sin (2 \pi  x_1) \sin (\pi  x_2)),\\
(\lambda_3,u_3)&=(\lambda_{(1,2)},u_{(1,2)})=(5 \pi ^2,\sin (\pi  x_1) \sin (2 \pi x_2)),\\
(\lambda_4,u_4)&=(\lambda_{(2,2)},u_{(2,2)})=(8 \pi ^2,\sin (2 \pi  x_1) \sin (2 \pi  x_2)),
\end{align*}
where the eigenpairs are indexed also by the Fourier indices. The double eigenvalue is due to symmetry as shown in Figure~\ref{fig:squarevecs}. Here we are interested in the case where the diffusion coefficient is no longer constant but depends on a countable number of parameters. For example, we could think of a stochastic coefficient given in the form of a Karhunen-Lo\`eve expansion. If the variation in the diffusion is restricted to $x_1$-direction, say, it is intuitively clear that it follows within the cluster $(u_{(2,1)},u_{(1,2)})$ that $\lambda_{(2,1)} \neq \lambda_{(1,2)}$. Indeed, the relative order of the eigenvalues $\lambda_{(2,1)}$ and $\lambda_{(1,2)}$ ultimately depends on the realisations of the diffusion parameters. This is the mechanism that induces the crossing of eigenvalues in this context. Moreover, in the general setting it has to be established under what assumptions the cluster itself, i.e., the eigenspace associated to $\lambda_{(2,1)}$ and $\lambda_{(1,2)}$ remains isolated. 
\end{exm}

\begin{figure}[htb]
\begin{center}
\subfloat[{$u_1 = u_{(1,1)}$.}]
{\includegraphics[width=0.22\textwidth]{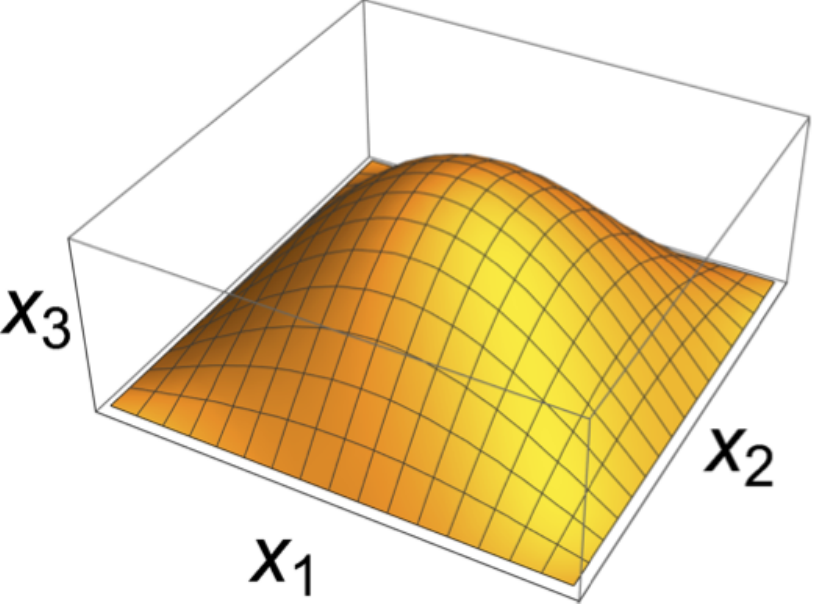}}\ 
\subfloat[{$u_2 = u_{(2,1)}$.}]
{\includegraphics[width=0.22\textwidth]{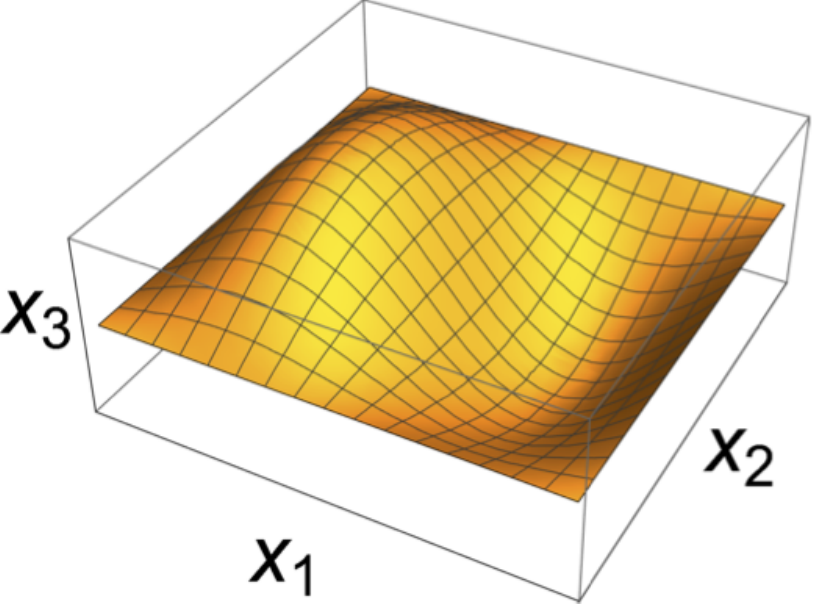}}\ 
\subfloat[{$u_3 = u_{(1,2)}$.}]
{\includegraphics[width=0.22\textwidth]{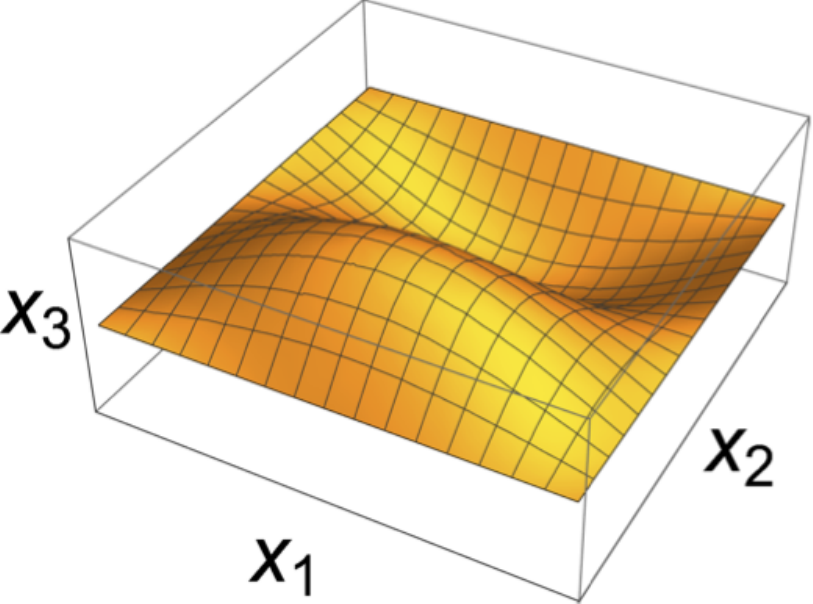}}\ 
\subfloat[{$u_4 = u_{(2,2)}$.}]
{\includegraphics[width=0.22\textwidth]{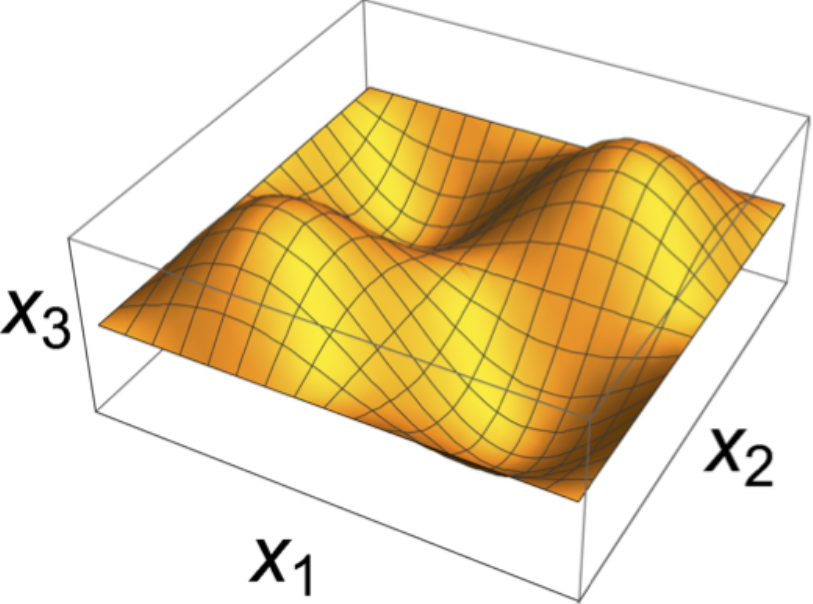}}
\caption{Dirichlet Laplacian in the unit square: First four eigenfunctions.}
\label{fig:squarevecs}
\end{center}
\end{figure}

In this paper we consider eigenspaces of an elliptic self-adjoint operator that depends affinely on a countable number of parameters. Our main theoretical contribution is that we extend the results in \cite{andreevschwab12} on analyticity to cover eigenspaces associated to possibly clustered eigenvalues. The underlying assumption is that the eigenspace of interest is isolated in the sense that the corresponding eigenvalues are separated from the rest of the spectrum for all values of the input parameters. We show that the spectral projection operator associated to such an isolated eigenspace can in fact be extended to a complex-analytic function of the input parameters. This allows us to construct a well-defined and smooth basis for the eigenspace of interest and show that optimal convergence rates hold when the basis vectors are approximated using a conveniently chosen set of orthogonal polynomials. We consider the stochastic collocation method defined on an anisotropic sparse grid in the parameter domain, similar to the one in \cite{andreevschwab12}, for computing a basis for the eigenspace of interest. Our numerical experiments show that optimal rates of convergence hold even in the presence of eigenvalue crossings. In fact, in our examples we observe fast rates of convergence even if the terms in the Karhunen- Lo\`eve series decay too slowly for the current theory to hold.

Our method constructs a basis for the eigenspace. This can be useful for at least two reasons. First of all, once the basis has been computed, we may project our original eigenproblem on this basis. It then becomes easier to track the individual eigenpairs as we no longer need to deal with the original full eigenvalue problem. Instead, the actual eigeninformation, values and modes, could be recovered, for instance, by sampling realisations of the eigenmodes of the projected problem. Second of all, in important applications such as frequency response analysis, finding a representation for the eigenspace may be of independent interest, and here it is obtained directly.

The rest of this paper is structured as follows: In Section~\ref{sec:problem}
the model problem is defined in its multiparametric form, the main result, analyticity of isolated eigenspaces is established in Section~\ref{sec:analyticity},
the collocation scheme is defined in Section~\ref{sec:collocation}, and
the numerical experiments in Section~\ref{sec:numerical}, before concluding remarks in Section~\ref{sec:conclusions}.

\section{Problem formulation}
\label{sec:problem}

We consider a class of self-adjoint operators that depend affinely on a countable number of real parameters. This affine dependence is often of independent interest but may also result from first order approximation of more general smooth dependence. In particular, the commonly used model problem for a stochastic diffusion operator falls within our framework.

\subsection{Multiparametric variational eigenvalue problems}

Let $V$ and $H$ be separable Hilbert spaces over $\R$ and denote the associated inner products by $(\cdot, \cdot)_V$ and $(\cdot, \cdot)_H$ and norms by $\norm{\ \cdot \ }{V}$ and $\norm{\ \cdot \ }{H}$. Assume that $V$ and $H$ form the so-called Gel'fand triple $V \subset H \subset V^*$ with dense and compact embeddings. We denote by $\mc{L}(V,V^*)$ the space of bounded linear operators from $V$ to its dual $V^*$. Furthermore, we denote by $\langle \cdot, \cdot \rangle_{V \times V^*}$ the duality pairing on $V$ and $V^*$, which may be interpreted as an extension of the inner product $(\cdot, \cdot)_H$.

For each $m \in \N_0$ let $b_m \! : V \times V \to \R$ be a symmetric and continuous bilinear form, which we can associate with an operator $B_m \in \mc{L}(V,V^*)$ using
\[
b_m(u,v) = \langle v, B_m u \rangle_{V \times V^*} \quad \forall u,v \in V.
\]
Suppose that there exists $\alpha_0 > 0$ such that
\begin{equation}
\label{eq:ellipticity}
b_0(v,v) \ge \alpha_0 \norm{v}V^2 \quad \forall v \in V
\end{equation}
and a sequence $\kappa = (\kappa_1, \kappa_2, \ldots)$ of positive real numbers such that $\norm{\kappa}{\ell^1(\N)} < 1$ and
\begin{equation}
\label{eq:decay}
|b_m(u,v)| \le \kappa_m \alpha_0 \norm{u}{V} \norm{v}{V} \quad \forall u,v \in V.
\end{equation}
We define a multiparametric bilinear form
\begin{equation}
\label{eq:bilinearform}
b(y;u,v) := b_0(u,v) + \sum_{m=1}^{\infty} y_m b_m(u,v), \quad u, v \in V,
\end{equation}
where $y = (y_1, y_2, \ldots)$ is a vector of parameters, each of which takes values in a closed interval of $\R$. Equivalently, we can call \eqref{eq:bilinearform} a multiparametric family of bilinear forms. Without loss of generality we may assume a scaling such that $y \in \Gamma := [-1,1]^{\infty}$. We associate the form \eqref{eq:bilinearform} with a multiparametric family of operators $B \! : \Gamma \to \mc{L}(V,V^*)$ given by
\begin{equation}
\label{eq:mpoperator}
B(y) := B_0 + \sum_{m=1}^{\infty} y_m B_m.
\end{equation}

\begin{rmk}
The ellipticity condition \eqref{eq:ellipticity} could be weakened by assuming that
\[
b_0(v,v) + w \norm{v}{H}^2 \ge \alpha_0 \norm{v}{V}^2 \quad \forall v \in V
\]
for some $w > 0$ and $\alpha_0 > 0$. This can be reduced to the elliptic case using a standard shift procedure.
\end{rmk}

The assumptions above imply that $b(y;\cdot,\cdot)$ is uniformly bounded and uniformly elliptic, i.e.,
\[
\sup_{y \in \Gamma} |b(y;u,v)| \le C\norm{u}V\norm{v}V \quad \forall u,v \in V
\]
and
\[
\inf_{y \in \Gamma} b(y;v,v) \ge \alpha \norm{v}V^2 \quad \forall v \in V
\]
for some $C > 0$ and $\alpha > 0$. Consider the following multiparametric eigenvalue problem: find $\mu \! : \Gamma \to \R$ and $u \! : \Gamma \to V \backslash \{ 0 \}$ such that
\begin{equation}
\label{eq:pevp}
B(y)u(y) = \mu(y)u(y),
\end{equation}
or in variational form
\begin{equation}
\label{eq:pevpvarform}
b(y;u(y),v) = \mu(y) (u(y),v)_H \quad \forall v \in V.
\end{equation}
The Lax-Milgram lemma guarantees that for any $y \in \Gamma$ the operator $B(y)$ is boundedly invertible and its inverse $B^{-1}(y) \! : H \to V$ is compact due to the compact embedding $V \subset H$. Therefore, the problem admits a countable number of real eigenvalues of finite multiplicity and associated eigenfunctions that form an orthogonal basis of $H$.

\begin{rmk}
\label{rmk:modelproblem}
A commonly used model problem is the stochastic diffusion eigenvalue problem on $D \subset \R^n$
\begin{equation}
\label{eq:modelproblem}
\left\{ \begin{array}{ll} - \nabla \cdot (a(\cdot, y) \nabla u(\cdot,y)) = \mu(y) u(\cdot, y) & \textrm{in }D \\
u(\cdot, y) = 0 & \textrm{on }\partial D, \end{array} \right.
\end{equation}
where the diffusion coefficient is a random field expressed in its Karhunen-Lo\`eve expansion
\begin{equation}
\label{eq:klexpansion}
a(x,y) = a_0(x) + \sum_{m=1}^{\infty} a_m(x) y_m, \quad x \in D, \quad y \in \Gamma.
\end{equation}
Indeed (if $D$ is nice enough) the variational formulation of \eqref{eq:modelproblem} is given by \eqref{eq:pevpvarform} with the choice $V = H^1_0(D)$, $H = L^2(D)$ and
\[
b_m(u,v) = \int_D a_m \nabla u \cdot \nabla v \ dx \quad \forall u,v \in V, \quad m \in \N_0.
\]
Let $C_D$ denote the Poincar\'e constant for $D$. It is now easy to see that the inequalities \eqref{eq:ellipticity} and \eqref{eq:decay} are satisfied if $\alpha_0 := (1 + C_D)^{-1} \essinf_{x \in D} a_0(x) > 0$ and $\kappa_m := \alpha_0^{-1}\norm{a_m}{L^{\infty}(D)}$ are such that $\norm{\kappa}{\ell^1(\N)} < 1$.
\end{rmk}

We will assume an increasing enumeration of the eigenvalues so that
\[
0 < \mu_1(y) \le \mu_2(y) \le \ldots \quad \forall y \in \Gamma,
\]
where each eigenvalue may be listed several times according to its multiplicity. We denote by $\{ u_i(y) \}_{i=1}^{\infty}$ a set of associated eigenfunctions which are orthonormal in $H$ for every $y \in \Gamma$. Ultimately we would like to compute any given subset of the eigenpairs $\{ (\mu_i, u_i) \}_{i=1}^{\infty}$ of problem \eqref{eq:pevp}. However, due to possible eigenvalue crossings, this may sometimes be an extremely difficult task to perform computationally, see e.g. \cite{hakulalaaksonen19a}, \cite{hakulalaaksonen19b}. Therefore, we will work under the assumption that the eigenspace of interest is isolated, i.e., the associated eigenvalues are strictly separated from the rest of the spectrum.

\subsection{Isolated eigenspaces}

Let $J \subset \N$ and $S = \#J$ denote its cardinality. For $y \in \Gamma$ let $\sigma_J(y) := \{ \mu_i(y) \}_{i \in J}$ denote a set of eigenvalues of the problem \eqref{eq:pevp} and $U_J(y) := \spa \{u_i(y)\}_{i \in J}$ denote the associated eigenspace. We use a shorthand notation $U_S$ for the eigenspace $U_J$ with $J = \{1,2, \ldots, S\}$. We call an eigenspace $U_J$ isolated with parameter $\delta > 0$ (or simply just isolated) if
\[
\dist(\sigma_J(y), \sigma_{\N \backslash J}(y)) \ge \delta \max \sigma_J(y) \quad \forall y \in \Gamma.
\]
A set of functions $\{g_i\}_{i=1}^S \subset V^{\Gamma}$ is called a basis of $U_J$ if
\[
U_J(y) = \spa \{g_i(y)\}_{i=1}^S \quad \forall y \in \Gamma.
\]
Moreover, this basis is called orthonormal if $\{g_i(y)\}_{i=1}^S$ is orthonormal in $H$ for every $y \in \Gamma$. In the context of this paper we are interested in computing a basis for a given isolated eigenspace $U_J$. We aim to demonstrate that, though the set of eigenvectors $\{u_i\}_{j \in J}$ clearly is an orthonormal basis of $U_J$, it may not always be computationally the most accessible one.

\begin{rmk}
Note that even if the eigenspace $U_J$ is isolated, double eigenvalues or eigenvalue crossings may still exist within the set $\{\mu_i\}_{i \in J}$. In other words, we might have $\mu_i(y) = \mu_j(y)$ and $i \not= j$ for some $i, j \in J$ and $y \in \Gamma$.
\end{rmk}

The following is an adaptation of the classical theorem by Weyl.

\begin{prp}
\label{prp:weyl}
Under assumptions \eqref{eq:ellipticity} and \eqref{eq:decay} the eigenvalues of the problem \eqref{eq:pevp} satisfy
\[
\left( 1 - \norm{\kappa}{\ell^1(\N)} \right) \mu_i(0) \le \mu_i(y) \le \left( 1 + \norm{\kappa}{\ell^1(\N)} \right) \mu_i(0), \quad i \in \N, \quad y \in \Gamma.
\]
\end{prp}

\begin{proof}
Recall the min-max characterization of eigenvalues. For $i \in \N$ let $V^{(i)}$ denote the set of all subspaces of $V$ with dimension equal to $i$. Given a subspace $U \subset V$ we set $\widehat{U} = \{ v \in U \ | \ \norm{v}{H} = 1 \}$. For some $u \in \widehat{U}_i(y)$ we now have
\[
\mu_i(0) = \min_{U \in V^{(i)}} \max_{v \in \widehat{U}} \ b_0(v,v) \le \max_{v \in \widehat{U}_i(y)} b_0(v,v) = b_0(u,u)
\]
and
\[
\mu_i(y) = \min_{U \in V^{(i)}} \max_{v \in \widehat{U}} \ b(y;v,v) = \max_{v \in \widehat{U}_i(y)} b(y;v,v) \ge b(y;u,u).
\]
It follows that
\[
\mu_i(y) \ge b(y;u,u) \ge \left( 1 - \norm{\kappa}{\ell^1(\N)} \right) b_0(u,u) \ge \left( 1 - \norm{\kappa}{\ell^1(\N)} \right) \mu_i(0).
\]
Similarly for some $u \in U_i(0)$ we have
\[
\mu_i(y) = \min_{U \in V^{(i)}} \max_{v \in \widehat{U}} \ b(y;v,v) \le \max_{v \in \widehat{U}_i(0)} b(y;v,v) = b(y;u,u)
\]
and
\[
\mu_i(0) = \min_{U \in V^{(i)}} \max_{v \in \widehat{U}} \ b_0(v,v) = \max_{v \in \widehat{U}_i(0)} b_0(v,v) \ge b_0(u,u)
\]
so that
\[
\mu_i(y) \le b(y;u,u) \le \left( 1 + \norm{\kappa}{\ell^1(\N)} \right) b_0(u,u) \le \left( 1 + \norm{\kappa}{\ell^1(\N)} \right) \mu_i(0).
\]
\end{proof}

As a corollary we obtain sufficient criteria for an eigenspace to be isolated. For simplicity we state these only in the case of an eigenspace $U_S$ with $S \in \N$.

\begin{cor}
\label{cor:separated}
Assume \eqref{eq:ellipticity} and \eqref{eq:decay}. Given $S \in \N$ let
\begin{equation}
\label{eq:separationcriterion}
\delta_0 := \frac{\mu_{S+1}(0) - \mu_S(0)}{\mu_S(0)} > \frac{2}{\norm{\kappa}{\ell^1(\N)}^{-1}-1}.
\end{equation}
Then the eigenspace $U_S$ of the problem \eqref{eq:pevp} is isolated with parameter
\[
\delta = \delta(\delta_0, \kappa) := \frac{\delta_0 - (\delta_0+2)\norm{\kappa}{\ell^1(\N)}}{1 + \norm{\kappa}{\ell^1(\N)}} > 0.
\]
\end{cor}

\begin{proof}
Clearly $\delta > 0$. By Proposition \ref{prp:weyl} we have
\begin{align*}
\mu_{S+1}(y) - \mu_S(y) & \ge \left(1 - \norm{\kappa}{\ell^1(\N)} \right)\mu_{S+1}(0) - \left(1 + \norm{\kappa}{\ell^1(\N)} \right)\mu_S(0) \\
& = \delta \left( 1 + \norm{\kappa}{\ell^1(\N)} \right) \mu_S(0) \\
& \ge \delta \mu_S(y)
\end{align*}
for all $y \in \Gamma$.
\end{proof}

\subsection{Canonical bases}
\label{cb}

Given a set $J \subset \N$ with cardinality $S$, we define a canonical basis $\{ \hat{u}_i \}_{i=1}^S$ for the eigenspace $U_J$ by setting
\[
\hat{u}_i(y) = \sum_{j \in J} (u_{J(i)}(0),u_j(y))_H \ \! u_j(y) \quad \forall y \in \Gamma.
\]
Here $J(i)$ denotes the $i$th element in any fixed permutation of $J$. Observe that the canonical basis vectors $\{ \hat{u}_i \}_{i=1}^S$ now only depend on the eigenspace $U_J$ and not on the choice of the individual eigenvectors $\{ u_i \}_{i \in J}$. Moreover, if the matrix $\{ (u_{J(i)}(0),u_{J(j)}(y))_H \}_{i,j=1}^S$ is nonsingular, then $\{ \hat{u}_i \}_{i=1}^S$ is in fact a basis for $U_J$. Note that $\hat{u}_i(y)$ need not be orthonormal for $y\ne 0$ and that the inverse of the lowermost singular value of the Gram matrix $\{ (u_{J(i)}(0),u_{J(j)}(y))_H \}_{i,j=1}^S$ denotes the condition number of the basis and is uniformly bounded away from infinity due to the spectral separation assumption as given by the standard results for the convergence radii for the perturbation expansions of spectral projections from \cite{kato76}.

\section{Analyticity of isolated eigenspaces}
\label{sec:analyticity}

Next we will prove that any isolated eigenspace is in fact analytic with respect to the parameter vector $y \in \Gamma$ in a suitable sense. To this end we extend our analysis for complex valued arguments: In this section we assume that $V$ and $H$ are separable Hilbert spaces over $\C$ and extend the inner products $(\cdot, \cdot)_V$ and $(\cdot, \cdot)_H$ as well as the duality pairing $\langle \cdot, \cdot \rangle_{V \times V^*}$ for complex-valued arguments sesquilinearly. Now \eqref{eq:mpoperator} can be treated as the restriction to $\Gamma$ of the operator-valued function
\[
B(z) = B_0 + \sum_{m=1}^{\infty} z_m B_m, \quad z \in \C^{\infty}.
\]
We equip $\Gamma \subset \C^{\infty}$ with the Hausdorff topology so that this fits the framework of \cite{herve89}.

\subsection{Riesz spectral projection}

For $z \in \C^{\infty}$ let $\Omega(z)$ be a closed curve in the complex plane, which encloses a set of eigenvalues of $B(z)$, denoted by $\sigma_J(z)$, but no other elements in the spectrum of $B(z)$. We define the spectral projection
\[
P_J(z) = \frac{1}{2 \pi \iu} \int_{\Omega(z)} (\omega - B(z))^{-1} d \omega.
\]
We call the mapping $z \mapsto U_J(z)$ analytic if $z \mapsto P_J(z)$ is analytic, i.e., the mapping $z \mapsto (P_J(z)v,u)_H$ is analytic for all $v,u \in V$. Note that $z \mapsto (P_J(z)v,u)_H$ is a standard complex function of a complex variable. 

The canonical basis from section \ref{cb} can now be expressed as
$$
\hat{u}_i(y)=P_J(y)u_{J(i)}(0).
$$

\subsection{Analyticity in one parameter}

We first restrict our analysis to operators depending on a single parameter. In other words we consider the eigenvalues of \eqref{eq:pevp} when $y \in \Gamma$ is replaced by $t \in [-1,1]$ and our operator thus takes the form
\begin{equation}
\label{eq:singleparameteroperator}
B(t) = B_0 + t B_1, \quad t \in [-1,1].
\end{equation}
Here \eqref{eq:singleparameteroperator} will be understood as the restriction to $[-1,1]$ of the operator-valued function
\[
B(z) = B_0 + z B_1, \quad z \in \C.
\] 
The assumptions \eqref{eq:ellipticity} and \eqref{eq:decay} now imply
\begin{equation}
\label{eq:coercive}
\langle v, B_0 v \rangle_{V \times V^*} \ge \alpha_0 \norm{v}{V}^2, \quad \forall v \in V
\end{equation}
and
\begin{equation}
\label{eq:bounded}
\norm{B_1}{\mc{L}(V,V^*)} \le \kappa_1 \alpha_0
\end{equation}
for some $\alpha_0 > 0$ and $0 < \kappa_1 < 1$. We obtain the following result.

\begin{prp}
\label{prp:1danalyticity}
Consider the problem \eqref{eq:pevp} with $B_m = 0$ for $m \ge 2$, i.e., $B \! : [-1,1] \to \mc{L}(V,V^*)$ is of the form \eqref{eq:singleparameteroperator} and satisifies \eqref{eq:coercive} and \eqref{eq:bounded}. Given a finite $J \subset \N$ assume that the eigenspace $t \to U_J(t)$ is isolated with parameter $\delta > 0$ for $t \in [-1,1]$. Then it admits a complex-analytic extension $z \to U_J(z)$ to the region
\[
E(r) := \{ z \in \C \ | \ \exists t \in [-1,1] \ \mathrm{s.t.} \ |z-t| < r(t) \},
\]
where
\[
r(t) := \frac{\kappa_1^{-1}-|t|}{2(1+\delta^{-1})}.
\]
Moreover, for every $z \in E(r)$ the spectrum of $B(z)$ is separated into two parts $\sigma_J(z)$ and $\sigma_{\N \backslash J}(z)$ such that $\dist(\sigma_J(z), \sigma_{\N \backslash J}(z)) > 0$. 
\end{prp}

\begin{proof}
Assume first that $J$ is a set of consecutive natural numbers. Let $t \in [-1,1]$ and denote $\gamma(t) := \dist(\sigma_J(t),\sigma_{\N \backslash J}(t)) > 0$. Let $\Omega(t)$ be the positively oriented circle of radius
\[
\rho(t) = \frac{1}{2}( \max \sigma_J(t) - \min \sigma_J(t)) + \frac{\gamma(t)}{2}
\]
centered at
\[
c(t) = \frac{1}{2} (\max \sigma_J(t) + \min \sigma_J(t)).
\]
Then $\Omega(t)$ encloses $\sigma_J(t)$ but no elements of $\sigma_{\N \backslash J}(t)$. Moreover, for every $\omega \in \Omega(t)$ we have
\begin{align*}
\norm{B(t) (B(t) - \omega)^{-1}}{\mc{L}(V,V^*)} & = \norm{\mathrm{id} + \omega (B(t) - \omega)^{-1}}{\mc{L}(V,V^*)} \\
& \le 1 + |\omega| \norm{(B(t) - \omega)^{-1}}{\mc{L}(V,V^*)} \\
& \le 1 + \left(\max \sigma_J(t) + \frac{\gamma(t)}{2} \right) \left( \frac{\gamma(t)}{2} \right)^{-1} \\
& = 2 \left(1 + \frac{\max \sigma_J(t)}{\gamma(t)} \right) \\
& \le 2 (1 + \delta^{-1}).
\end{align*}
Due to \eqref{eq:coercive} and \eqref{eq:bounded} we have
\[
\norm{B(t)v}{V^*} \ge \norm{B_0v}{V^*} - |t| \norm{B_1v}{V^*} \ge \alpha_0(1 - \kappa_1 |t|) \norm{v}{V}
\]
so that
\[
\norm{B_1v}{V^*} \le \kappa_1 \alpha_0 \norm{v}{V} \le \frac{\kappa_1}{1 - \kappa_1 |t|} \norm{B(t)v}{V^*}
\]
for all $v \in V$. By Remark VII.2.9 in \cite{kato76} there exists $r_0(t) > 0$ such that whenever $|z - t| < r_0(t)$ the spectrum of $B(z)$ is separated into two parts $\sigma_J(z)$ and $\sigma_{\N \backslash J}(z)$ by the curve $\Omega(t)$. Moreover, for such values of $z$ the spectral projection valued function $z \mapsto P_J(z)$ is complex-analytic. In fact we may set $a = c = 0$ and $b = \kappa_1(1 - \kappa_1 |t|)^{-1}$ in the definition of $r_0(t)$ and obtain
\begin{align*}
r_0(t) \ge \left( \frac{2(1+\delta^{-1}) \kappa_1}{1 - \kappa_1 |t|} \right)^{-1} = \frac{\kappa_1^{-1} - |t|}{2(1+\delta^{-1})}.
\end{align*}
Since $t \in [-1,1]$ was arbitrary we conclude that $z \to P_J(z)$ is complex-analytic in $E(r)$.

An arbitrary $J \subset \N$ may always be partitioned in such a way that each partition is a set of consecutive natural numbers. The previous proof applies for all partitions separately and thus the spectrum of $B(z)$ is separated for all $z \in E(r)$ and the total projection $z \mapsto P_J(z)$ is complex-analytic in $E(r)$.
\end{proof}

\subsection{Analyticity in a countable number of parameters}

We start with a simple Lemma that can be deduced from standard perturbation theory for analytic operators, see Chapter VII in \cite{kato76}.

\begin{lmm}
\label{lmm:separately}
Let $z \in \C^{\infty}$ and $J \subset \N$ be such that the spectrum of $B(z)$ can be separated into two parts $\sigma_J(z)$ and $\sigma_{J \backslash \N}(z)$ with $\dist(\sigma_J(z), \sigma_{J \backslash \N}(z)) > 0$. Let $m \in \N$ and $e_m$ denote the $m$:th unit vector in $\R^{\infty}$. Then there exists $\epsilon(z) > 0$ such that the eigenspace $\zeta \to U_J(z + e_m \zeta)$ is complex-analytic for all $\zeta \in \C$ such that $|\zeta| < \epsilon(z)$. 
\end{lmm}

Suppose now that $\kappa \in \ell^p(\N)$ for some $p \in (0,1]$. Then we have the following result.

\begin{thm}
\label{thm:analyticity}
Consider the problem \eqref{eq:pevp} with assumptions \eqref{eq:ellipticity} and \eqref{eq:decay}. Assume that $\kappa \in \ell^p(\N)$ for some $p \in (0,1]$. Given a finite $J \subset \N$ assume that the eigenspace $y \to U_J(y)$ is isolated with parameter $\delta > 0$ for $y \in \Gamma$. Then it admits a complex-analytic extension $z \to U_J(z)$ in the region
\[
E(\tau) := \{ z \in \C^{\infty} \ | \ \dist(z_m, [-1,1]) < \tau_m \},
\]
where $\tau = (\tau_1, \tau_2, \ldots)$ is given by
\[
\tau_m := \frac{(1-\varepsilon)(1 - \norm{\kappa}{\ell^1(\N)}) \kappa_m^{p-1}}{2 \norm{\kappa}{\ell^p(\N)}(1+\delta^{-1})}, \quad m \in \N
\]
and $\varepsilon \in (0,1)$ is arbitrary.
\end{thm}

\begin{proof}
Let $z \in E(\tau)$ and take $y \in \Gamma$ such that $| z_m - y_m| < \tau_m$ for all $m \ge 1$. Denote $\zeta := z - y$. We now have
\[
\langle v, B(y) v \rangle_{V \times V^*} \ge \alpha_0 (1 - \norm{\kappa}{\ell^1(\N)}) \norm{v}{V}^2 \quad \forall v \in V
\]
and
\[
\norm{\sum_{m=1}^{\infty} \zeta_m B_m}{\mc{L}(V,V^*)} \le \sum_{m=1}^{\infty} \tau_m \norm{B_m}{\mc{L}(V,V^*)} \le \alpha_0 \sum_{m=1}^{\infty} \tau_m \kappa_m \le \alpha_0 (1 - \norm{\kappa}{\ell^1(\N)}) \tilde{\kappa},
\]
where
\[
\tilde{\kappa} := \sum_{m=1}^{\infty} \frac{\tau_m \kappa_m}{1 - \norm{\kappa}{\ell^1(\N)}} = \frac{1- \varepsilon}{2 (1+\delta^{-1})} < 1.
\]
Proposition \ref{prp:1danalyticity} now applies for the shifted operator
\[
t \mapsto B(y + t \zeta) = B(y) + t \sum_{m=1}^{\infty} \zeta_m B_m
\]
and therefore the associated eigenspace $t \mapsto U_J(y + t\zeta)$ can be extended to a function $\tilde{z} \mapsto U_J(y + \tilde{z} \zeta)$ which is complex-analytic for all $\tilde{z} \in \C$ such that
\[
|\tilde{z}| < \frac{\tilde{\kappa}^{-1}}{2(1+\delta^{-1})} = (1- \varepsilon)^{-1} > 1.
\]
In particular the eigenspace $\tilde{z} \mapsto U_J(y + \tilde{z}\zeta)$ is analytic in the vicinity of $\tilde{z} = 1$. By Lemma \ref{lmm:separately} the eigenspace $U_J$ is now separately complex-analytic in the vicinity of $z$. Since $z \in E(\tau)$ was arbitrary, we see that the eigenspace is separately complex-analytic in $E(\tau)$. Therefore, we may take Hartogs's theorem (Theorem 2.2.8 in \cite{hormander66}) and extend it to infinite dimensions (Definition 2.3.1, Proposition 3.1.2 and Theorem 3.1.5 in \cite{herve89}) to see that the eigenspace is jointly complex-analytic in $E(\tau)$.
\end{proof}

\section{Stochastic collocation on sparse grids}
\label{sec:collocation}

In the following we formulate a class of anisotropic sparse grid collocation operators defined with respect to finite and monotone multi-index sets. Our method falls into the abstract framework of \cite{andreevschwab12}, which in turn generalizes the collocation methods introduced earlier in e.g. \cite{nobiletemponewebster08a} and \cite{nobiletemponewebster08b}.

\subsection{General multi-index collocation}

We start by defining standard one-dimensional Lagrange interpolation operators, and then extend these to multiple dimensions in a sparse fashion. The interpolation points are chosen to be zeros of orthogonal (Legendre) polynomials.

Let $L_p$ denote the univariate Legendre polynomial of degree $p \in \N_0$, $\{ \chi_k^{(p)} \}_{k=0}^p$ denote the abscissae of $L_{p+1}$ and $\{ w_k^{(p)} \}_{k=0}^p$ denote the associated Gauss-Legendre quadrature weights. We define one-dimensional interpolation operators $\mc{I}^{(m)}_p$ which map a function $f \in C([-1,1])$ to the unique polynomial of degree $p$ that interpolates $f$ at the points $\{ \chi_k^{(p)} \}_{k=0}^p$. This may be written in Lagrange form as
\begin{equation}
\label{eq:colloperator1d}
\left(\mc{I}^{(m)}_p f \right) \! (y_m) = \sum_{k=0}^p f \!\left(\chi_k^{(p)} \right) \ell_k^{(p)} (y_m),
\end{equation}
where $\{ \ell_k^{(p)} \}_{k=0}^p$ are the standard Lagrange basis polynomials of degree $p$. We also have an alternative representation
\begin{equation}
\label{eq:legenrerepresentation1d}
\left(\mc{I}^{(m)}_p f \right) \! (y_m) = \sum_{k = 0}^p d_k(f) L_k(y_m), \quad p \in \N_0,
\end{equation}
where the coefficients $\{d_k\}_{k=0}^p$ are given by
\[
d_k(f) = \int_{-1}^1 \left(\mc{I}^{(m)}_p f \right) \! (y_m) L_k(y_m) \ \frac{d y_m}{2} = \sum_{j=0}^p f(\chi_j^{(p)}) L_k(\chi_j^{(p)}) w_j^{(p)}.
\]
This is due to the fact that Gauss-Legendre quadrature of order $p$ integrates any polynomial of degree $2p +1$ exactly. For more information we refer to \cite{davis75}.

Now let $(\mbb{N}_0^{\infty})_c$ denote the set of all multi-indices with finite support, i.e.,
\[
(\mbb{N}_0^{\infty})_c := \{ \a \in \mbb{N}_0^{\infty} \ | \ \# \supp(\a) < \infty \},
\]
where $\supp(\a) = \{ m \in \mbb{N} \ | \ \a_m \not= 0 \}$. Given a finite set $\mc{A} \subset (\mbb{N}_0^{\infty})_c$ we define the greatest active dimension $M_{\! \mc{A}} := \max \{ m \in \N \ | \ \exists \a \in \mc{A} \textrm{ s.t. } \a_m \not= 0 \}$. For $\a, \b \in \mc{A}$ we write $\a \le \b$ if $\a_m \le \b_m$ for all $m \ge 1$. We call the multi-index set $\mc{A}$ monotone if whenever $\b \in (\mbb{N}_0^{\infty})_c$ is such that $\b \le \a$ for some $\a \in \mc{A}$, then $\b \in \mc{A}$.

Given a finite and monotone set $\mc{A} \subset (\mbb{N}_0^{\infty})_c$ we define the sparse collocation operator
\begin{equation}
\label{eq:colloperator}
\mc{I}_{\mc{A}} := \sum_{\alpha \in \mc{A}} \bigotimes_{m \in \supp{\alpha}} \! \left( \mc{I}_{\alpha_m}^{(m)} - \mc{I}_{\alpha_m - 1}^{(m)} \right)
\end{equation}
with the convention $\mc{I}_{-1}^{(m)} := 0$. The operator \eqref{eq:colloperator} may be rewritten in a computationally more convenient form
\begin{equation}
\label{eq:interpolationreformulation}
\mc{I}_{\mc{A}} = \sum_{\alpha \in \mc{A}} c_{\a} \bigotimes_{m = 1}^{M_{\! \mc{A}}} \mc{I}_{\a_m}^{(m)}
\end{equation}
with coefficients
\[
c_{\a} := \! \sum_{\b \in \mc{A}} \mbf{1}_{\{\b - 1 \le \a \le \b\}} (-1)^{\norm{\alpha - \beta}{\ell^1({\N})}}.
\]
Note that aggregate quantities of our collocated solution may now be computed simply by applying Gauss-Legendre quadrature rules on the components of \eqref{eq:interpolationreformulation}.

By following Lemma 5 in \cite{andreevschwab12} we may express the collocated solution \eqref{eq:interpolationreformulation} as an expansion of multivariate Legendre polynomials. Given $\b, \g \in \mc{A}$ we define multi-dimensional collocation points
\[
\chi_{\g}^{(\b)} := \left(\chi_{\g_1}^{(\b_1)}, \ldots, \chi_{\g_{M_{\! \mc{A}}}}^{(\b_{M_{\! \mc{A}}})}, 0, 0, \ldots \right) \in \Gamma,
\]
associated multi-dimensional quadrature weights 
\[
w_{\g}^{(\b)}(y) := w_{\g_1}^{(\b_1)}(y_1) \cdots w_{\g_{M_{\! \mc{A}}}}^{(\b_{M_{\! \mc{A}}})}(y_{M_{\! \mc{A}}})
\]
and tensorized Legendre polynomials
\[
\Lambda_{\a}(y) := L_{\a_1}(y_1) \cdots L_{a_{M_{\! \mc{A}}}}(y_{M_{\! \mc{A}}}).
\]
Using the relation \eqref{eq:legenrerepresentation1d} we obtain
\begin{equation}
\label{eq:legendrerepresentation}
\left(\mc{I}_{\mc{A}} v \right) = \sum_{\alpha \in \mc{A}} c_{\a} \left( \bigotimes_{m = 1}^{M_{\! \mc{A}}} \mc{I}_{\a_m}^{(m)} \right) v = \sum_{\alpha \in \mc{A}} c_{\a} \sum_{\b \le \a} \Lambda_{\b} \sum_{\g \le \a} w_{\g}^{(\a)} \Lambda_{\b}(\chi_{\g}^{(\a)}) v(\chi_{\g}^{(\a)})
\end{equation}
so that
\[
\left( \mc{I}_{\mc{A}} v \right)(y) = \sum_{\a \in \mc{A}} d_{\a}(v) \Lambda_{\a}(y)
\]
with expansion coefficients given by
\[
d_{\a}(v) := \sum_{\b \in \mc{A}} \mbf{1}_{\{\b \ge \a\}} c_{\b} \sum_{\g \le \b} w_{\g}^{(\b)} \Lambda_{\a}(\chi_{\g}^{(\b)}) v(\chi_{\g}^{(\b)}).
\]
This expression is particularly convenient for evaluating our numerical solution in polynomial form. The number of collocation points required to evaluate the solution from equation \eqref{eq:legendrerepresentation} is
\[
N_{\mc{A}} := \# \{ \chi_{\g}^{(\a)} \in \Gamma \ | \ \a \in \mc{A}, \g \le \a \} = \sum_{\a \in \mc{A}} \prod_{m \in \supp{\a}} (\a_m + 1).
\]

\subsection{Convergence for a class of monotone multi-index sets}

The convergence rate of our collocation scheme depends on both the regularity of the solution at hand as well as the selection of the underlying multi-index sets. Here we follow the framework of \cite{bieriandreevschwab09} and \cite{andreevschwab12} and restrict ourselves to a particular choice of monotone multi-index sets. We then recapitulate the main convergence results from \cite{andreevschwab12}.

Given a sequence $\eta = (\eta_1, \eta_2, \ldots)$ such that $1 > \eta_1 \ge \eta_2 \ge \ldots \ge 0$ and $\eta_m \to 0$ we define the multi-index set
\begin{equation}
\label{eq:miset}
\mc{A}_{\varepsilon}(\eta) := \{ \a \in (\N^{\infty}_0)_c \ \vert \ \eta^{\a} \ge \varepsilon \}, \quad \varepsilon > 0
\end{equation}
where $\eta^{\a} := \prod_{m \in \supp{\a}} \eta_m^{\a_m}$ (with the convention $0^0 := 1$). The set $\mc{A}_{\varepsilon}(\eta)$ is now clearly finite and monotone. In view of Theorem \ref{thm:analyticity} we may set
\begin{equation}
\label{eq:eta}
\eta_m := \sup_{n \ge m} \frac{1}{\rho_n}, \quad m \in \N,
\end{equation}
where
\begin{equation}
\label{eq:rho}
\rho_m := \tau_m + \sqrt{1 + \tau^2_m}, \quad m \in \N
\end{equation}
is equal to the sum of the semiaxes of a Bernstein ellipse (see \cite{davis75} p. 19-20 and 312). We then obtain the following result.

\begin{prp}
\label{prp:convergence}
Let $H$ be a Hilbert space. Assume that $v \!: \Gamma \to H$ admits a complex-analytic extension in the region
\[
E(\tau) := \{ z \in \mbb{C}^{\infty} \ | \ \dist(z_m, [-1,1]) < \tau_m \}
\]
where $\tau = (\tau_1, \tau_2, \ldots)$ is a sequence of positive numbers such that $\tau_m \to \infty$. Define $\mc{A_{\varepsilon}}(\eta) \subset (\N_0^{\infty})_c$ according to \eqref{eq:miset}, \eqref{eq:eta} and \eqref{eq:rho}. Assume that $\eta_m m^{\sigma} \to 0$ for some $\sigma > 2(1 + \log 4)$. Then for any $1 > \varkappa > 2(1 + \log 4)/ \sigma$ there exists $C > 0$ such that
\[
\norm{v - \mc{I}_{\mc{A}_{\varepsilon}(\eta)}v}{L^2_{\nu}(\Gamma) \otimes H} \le C \varepsilon^{1 - \varkappa} \norm{v}{L^{\infty}(\overline{E(\tau)}; H)}
\]
for all $0 < \varepsilon \le \eta_1$. Here $\nu$ denotes the uniform probability measure on $\Gamma$. 
\end{prp}

\begin{proof}
Clearly there exists $M \ge 1$ such that
\[
\eta_m \le \eta_m' := (m+1)^{-\sigma} \quad \forall m > M.
\]
Lemma 7 and Proposition 3 in \cite{andreevschwab12} now imply that the so-called asymptotic overhead order of $\eta$ is
\[
\varkappa^*(\eta) \le \varkappa^*(\eta') \le 2(1 + \log 4)/\sigma.
\]
Hence, the result follows by taking $\varkappa > 2(1 + \log 4)/\sigma \ge \varkappa^*(\eta)$ in Theorem 6 (note also Remark 10) of \cite{andreevschwab12}.
\end{proof}

We may also estimate the convergence rate with respect to the number of collocation points.

\begin{thm}
\label{thm:convergencepoints}
Let the conditions of Proposition \ref{prp:convergence} hold. Then for any $s < \sigma - 2(1 + \log 4)$ there exists $C = C(v) > 0$ such that
\[
\norm{v - \mc{I}_{\mc{A}_{\varepsilon}(\eta)}v}{L^2_{\nu}(\Gamma) \otimes H} \le C(v) N_{\! \mc{A}_{\varepsilon}(\eta)}^{-s/2}
\]
as $\varepsilon \to 0$.
\end{thm}

\begin{proof}
By Lemma 6 in \cite{andreevschwab12} the cardinality of the multi-index sets $\mc{A}_{\varepsilon}(\eta)$ at the limit $\varepsilon \to 0$ is given by
\[
\# \mc{A}_{\varepsilon}(\eta) = F \left(\varepsilon^{-1/\sigma} \right), \quad \textrm{ where } F(x) = x \frac{e^{2\sqrt{\log x}}}{2 \sqrt{\pi}(\log x)^{3/4}}(1 + \mc{O}(1/\log x)).
\]
For any $\omega > 1$ we have the bound
\[
\# \mc{A}_{\varepsilon}(\eta) = F \left(\varepsilon^{-1/\sigma} \right) \lesssim \varepsilon^{-\omega/\sigma}
\]
and therefore $\varepsilon \lesssim (\# \mc{A}_{\varepsilon}(\eta))^{-\tilde{\sigma}}$ whenever $\tilde{\sigma} < \sigma$. Proposition \ref{prp:convergence} now implies that
\[
\norm{v - \mc{I}_{\mc{A}_{\varepsilon}(\eta)}v}{L^2_{\nu}(\Gamma) \otimes H} \lesssim (\# \mc{A}_{\varepsilon}(\eta))^{-\tilde{\sigma}(1 - \varkappa)}
\]
for any $1 > \varkappa > 2(1 + \log4)/\sigma$. Given $s < \sigma - 2(1 + \log 4)$ we may now choose $\varkappa > 2(1 + \log4)/\sigma$ and $\tilde{\sigma} < \sigma$ so that
\[
s < \tilde{\sigma}(1 - \varkappa)
\]
and therefore
\[
\norm{v - \mc{I}_{\mc{A}_{\varepsilon}(\eta)}v}{L^2_{\nu}(\Gamma) \otimes H} \lesssim (\# \mc{A}_{\varepsilon}(\eta))^{-s}.
\]
Finally, Lemma 4 in \cite{andreevschwab12} implies that $N_{\! \mc{A}_{\varepsilon}(\eta)} \le (\# \mc{A}_{\varepsilon}(\eta))^2$ and the claim follows.
\end{proof}

\subsection{Application to eigenspace computation}

In the following we briefly illustrate how the previous results can be applied to eigenspace computation in the particularly interesting case that the coefficients $\kappa = (\kappa_1, \kappa_2, \ldots)$ in \eqref{eq:decay} decay at an algebraic rate.

Suppose that the eigenspace $y \to U_J(y)$ of the problem \eqref{eq:pevp} is isolated for some finite $J \subset \N$. In addition to \eqref{eq:ellipticity} and \eqref{eq:decay} assume that
\[
\kappa_m \lesssim (m+1)^{-\varsigma}, \quad m \in \N,
\]
where $\varsigma-1 > \sigma > 2(1 + \log 4)$. Note that this implies $\kappa \in \ell^p(\N)$ for $p > \varsigma^{-1}$. By Theorem \ref{thm:analyticity} the eigenspace $y \to U_J(y)$ admits a complex analytic extension in the region
\[
E(\tau) := \{ z \in \mbb{C}^{\infty} \ | \ \dist(z_m, [-1,1]) < \tau_m \},
\]
where $\tau_m \gtrsim (m+1)^{\varsigma(1-p)}$. The sequence $\eta$ in \eqref{eq:eta} now converges at the rate
\[
\eta_m \lesssim \tau_m^{-1} \lesssim (m+1)^{-\varsigma(1-p)},
\]
where $\varsigma(1-p) > \sigma$ for sufficiently large $p$. Therefore $\eta_m m^{\sigma} \to 0$ and the conditions of Proposition \ref{prp:convergence} and Theorem \ref{thm:convergencepoints} hold. This means that we should expect the convergence rate
\begin{equation}
\label{eq:convergence}
\norm{\hat{u}_i - \mc{I}_{\mc{A}_{\varepsilon}}(\hat{u}_i)}{L^2_{\nu}(\Gamma) \otimes V} \le C(\hat{u}_i) N_{\! \mc{A}_{\varepsilon}(\eta)}^{-s/2}, \quad s > \varsigma - 3 - 2\log 4,
\end{equation}
when the sparse stochastic collocation algorithm is used to approximate the canonical basis vectors $\{ \hat{u}_i \}_{i \in J}$ of $U_J$.

\begin{rmk}
We may apply the Gram–Schmidt process at every collocation point in order to obtain an approximately orthonormal basis for $U_J$.
\end{rmk}

\section{Numerical examples: stochastic diffusion equation}
\label{sec:numerical}

In this section we present numerical examples to verify the convergence rate \eqref{eq:convergence} of our stochastic collocation algorithm. To this end we consider the model problem from Remark \ref{rmk:modelproblem}, i.e., the eigenvalue problem of a stochastic diffusion operator, and compute a canonical basis for one of its isolated eigenspaces. A standard finite element method is employed to obtain the discretization in physical space: In each of the examples the deterministic mesh is a grid of second order elements of diameter at most $h$, and the finite element space is then obtained by projecting the variational equation \eqref{eq:pevpvarform} onto the corresponding finite approximation space $V_h \subset V$. In the context of the current paper, however, we focus on the convergence in the parameter space and disregard the approximation error related to the spatial discretization.

Our numerical examples cover two different scenarios. First, we consider the model problem on the unit square and assume that the diffusion coefficient is constant in the second coordinate direction (Example~\ref{ex:model}). By separation of variables we may then either reduce this problem to a one-dimensional problem (in physical space), where each eigenvalue is well separated, or we may solve the full two-dimensional problem, where eigenmodes are tangled together. In particular we show that our subspace algorithm applied to the full two-dimensional problem converges to the same result as when a simple eigenvalue algorithm is employed to the dimensionally reduced problem. Second, we apply our algorithm to the model problem in a dumbbell shaped domain and let the diffusion coefficient depend on both spatial coordinates. In this case the crossing of eigenvalues is intrinsic by nature and the eigenmodes can not be untangled by mere separation of variables. We illustrate that similar convergence rates hold as in the first example.


\subsection{Reducible uncertainty model in the unit square}

Consider the stochastic diffusion problem from Remark \ref{rmk:modelproblem} with $D := (0,1)^2$. We let $a_0 := 1 + C_D$, where $C_D$ denotes the Poincar\'e constant for $D$, i.e., the inverse of the smallest eigenvalue of the Laplacian with Dirichlet boundary condition. For $m \in \N$ and $\varsigma > 1$ we set
\[
a_m(x) := (m + 1)^{-\varsigma} \sin(m \pi x_1), \quad x = (x_1, x_2) \in D.
\]
It is easy to see that the assumptions \eqref{eq:ellipticity} and \eqref{eq:decay} hold with $\alpha_0 = 1$ and $\kappa_m = (m+1)^{-\varsigma}$. For $\varsigma$ large enough, in particular for $\varsigma \ge 2$, we have $\norm{\kappa}{\ell^1(\N)} < 1$. In the following examples we have used the values $\varsigma = 3$ and $\varsigma = 6$. Moreover, we have used multi-index sets $\mc{A}_{\varepsilon}(\eta) \subset (\N_0^{\infty})_c$ as defined in the equations \eqref{eq:miset} -- \eqref{eq:rho}. The sequence $\tau = (\tau_1, \tau_2, \ldots)$ is set as $\tau_m := (m+1)^{\varsigma-1}$ which is in accordance with e.g. the numerical experiments in \cite{hakulalaaksonen19b}.

Since the diffusion coefficient $a(x) = a(x_1)$ is independent of $x_2$ we may now reduce the original problem to a one-dimensional problem. By separation of variables we see that functions of the form $u(x,y) = \varphi(x_1, y)\sin(\pi k x_2)$, where $k \in \N$ and $\varphi(x_1, y)$ solves
\begin{equation}
\label{eq:modelproblemreduced}
\left\{ \begin{array}{ll} - \partial_{x_1}(a(x_1, y) \partial_{x_1} \varphi(x_1, y)) + \pi^2 k^2 a(x_1, y) \varphi(x_1, y) = \lambda(y) \varphi(x_1, y), & x_1 \in (0,1) \\
\varphi(0, y) = \varphi(1, y) = 0 \end{array} \right.
\end{equation}
for all $y \in \Gamma$, form a complete set of eigenfunctions for our original problem \eqref{eq:modelproblem}. Classical Sturm-Liouville theory implies that the eigenvalues of \eqref{eq:modelproblemreduced} are simple and separated for every fixed $y \in \Gamma$. This separation of eigenvalues also holds uniformly with respect to $y \in U$, see Section 2.2 in \cite{gilbertgrahamkuoscheichlsloan19}. Hence, we may solve the eigenpairs of this one-dimensional problem via any simple stochastic eigenvalue algorithm such as the one presented in \cite{andreevschwab12}.

Let us now investigate the subspace $U_S$ of our model problem for $S = 3$. In Figure \ref{fig:squarecrossing} we have illustrated the first three eigenvalues of the problem as a function of the first parameter $y_1 \in [-1,1]$ when the rest are held constant. We see that there is an eigenvalue crossing at $y_1 = 0$ as is expected due to symmetry. As a result, there's multiple ways to choose the eigenfunction at this point. One example of the first three eigenfunctions at $y = 0$ has been shown in Figure \ref{fig:squarevecs}. By plugging the calculated eigenvalues and an estimate of $\norm{\kappa}{\ell^1(\N)}$ into \eqref{eq:separationcriterion}, we see that Corollary \ref{cor:separated} holds for $S = 3$ and the subspace $U_3$ is therefore isolated.

\begin{figure}[htb]
\begin{center}
\subfloat[{Eigenvalues $\{ \mu_i \}_{i=1}^4$.}]{\includegraphics[width=0.48\textwidth]{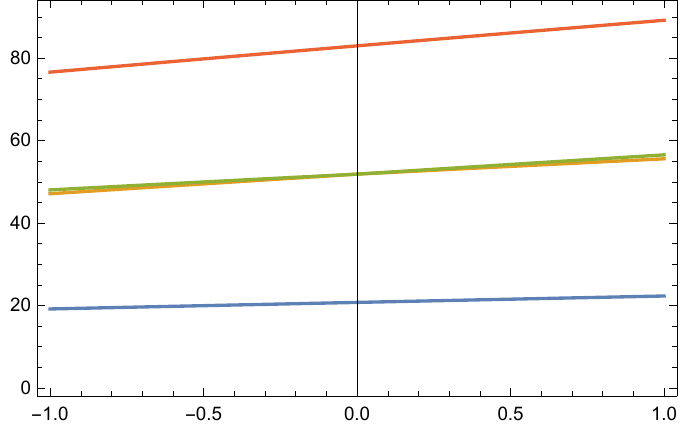}}\quad
\subfloat[{Eigenvalues $\{ \mu_i \}_{i=2}^3$.}]{\includegraphics[width=0.48\textwidth]{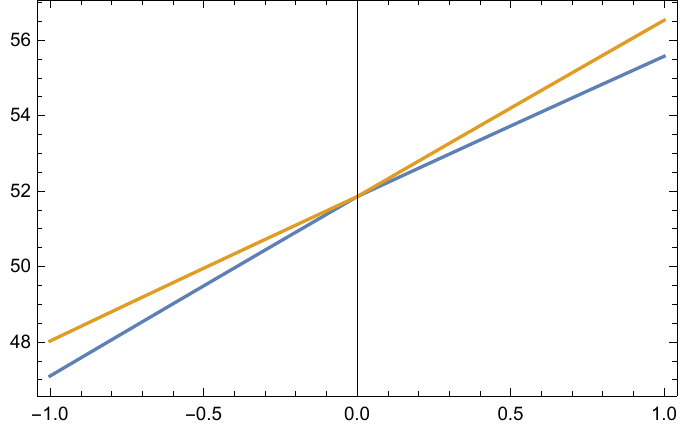}}
\caption{Model problem in the unit square with $\varsigma = 3$: First few eigenvalues as a function of $y_1 \in [-1,1]$ when $y_2 = y_3 = \ldots = 0$.}
\label{fig:squarecrossing}
\end{center}
\end{figure}

Let us first investigate the case $\varsigma = 6$ and employ our sparse stochastic collocation algorithm to compute a canonical basis for the subspace $U_3$. We compute a reference solution $\{ \hat{u}^*_i \}_{i=1}^3$ from the one-dimensional equation \eqref{eq:modelproblemreduced} using a mesh of 800 second order line elements. When computing this reference solution we set $\varepsilon > 0$ so that the number of multi-indices is $\# \mc{A}_{\varepsilon}(\eta) = 28$ and the greatest active dimension is $M_{\mc{A}_{\varepsilon}(\eta)} = 16$. This results in $N_{\mc{A}_{\varepsilon}(\eta)} = 77$ collocation points. Next we compute a series of solutions $\{ \hat{u}^{\varepsilon}_i \}_{i=1}^3$ from the two-dimensional equation \eqref{eq:modelproblem} using different values of $\varepsilon > 0$ and a mesh of 147456 second order quadrilateral elements. Convergence of the approximate basis $\{ \hat{u}^{\varepsilon}_i \}_{i=1}^3$ towards the reference solution $\{ \hat{u}^*_i \}_{i=1}^3$ with respect to the error measure
\[
\theta_{\epsilon} := \left( \sum_{i=1}^3 \norm{\hat{u}^{\varepsilon}_i - \hat{u}^*_i}{L^2_{\nu}(\Gamma) \otimes H^1(D)}^2 \right)^{1/2}
\]
has been illustrated in Figure \ref{fig:squareconv6}. The error behaves like $N_{\mc{A}_{\varepsilon}(\eta)}^{-3.0}$ with respect to the number of collocation points.

\begin{figure}[htb]
\begin{center}
\subfloat[{As a function of $\varepsilon$.}]{\includegraphics[width=0.48\textwidth]{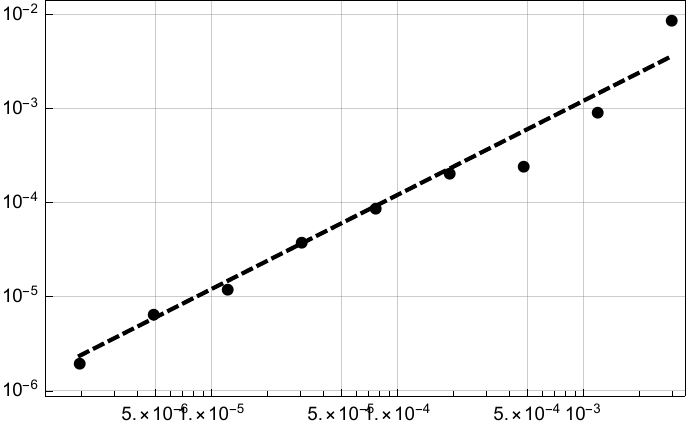}}\quad
\subfloat[{As a function of $N_{\mc{A}_{\varepsilon}(\eta)}$.}]{\includegraphics[width=0.48\textwidth]{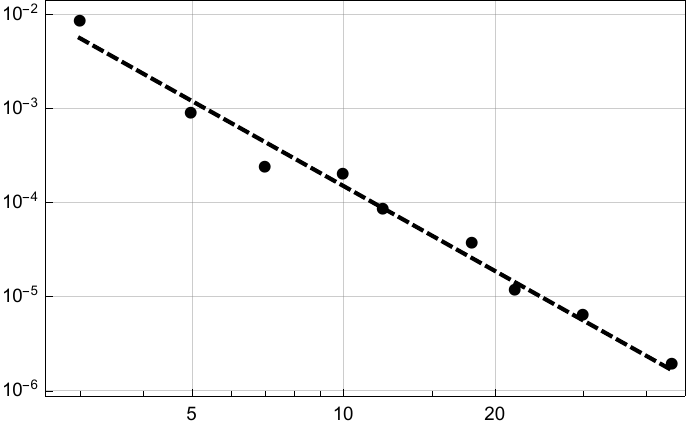}}
\caption{Model problem in the unit square when $\varsigma = 6$: Convergence of the approximate solution $\{ \hat{u}^{\varepsilon}_i \}_{i=1}^3$ to the reference solution $\{ \hat{u}^*_i \}_{i=1}^3$. The points represent values of the error measure $\theta_{\epsilon}$ on a log-log scale. Dashed lines represent algebraic rates $\varepsilon^{1.0}$ and $N_{\mc{A}_{\varepsilon}(\eta)}^{-3.0}$.}
\label{fig:squareconv6}
\end{center}
\end{figure}

We now repeat the previous exercise for $\varsigma = 3$. Note that Theorem \ref{thm:convergencepoints} does not in fact hold for this value of $\varsigma$ but numerically we observe convergence nevertheless. The faster rate of convergence of the terms in the Karhunen-Lo\`eve expansion \eqref{eq:klexpansion} justifies the use of a sparser discretization in physical space. In this case the reference solution $\{ \hat{u}^*_i \}_{i=1}^3$ is obtained from the one-dimensional equation \eqref{eq:modelproblemreduced} using a mesh of 160 second order line elements. For the reference solution we set $\varepsilon > 0$ so that the number of multi-indices is $\# \mc{A}_{\varepsilon}(\eta) = 302$ and the greatest active dimension is $M_{\mc{A}_{\varepsilon}(\eta)} = 129$, which gives us $N_{\mc{A}_{\varepsilon}(\eta)} = 1053$ collocation points. Again we compute a series of solutions $\{ \hat{u}^{\varepsilon}_i \}_{i=1}^3$ from the two-dimensional equation \eqref{eq:modelproblem} using different values of $\varepsilon > 0$ and a mesh of 6724 second order quadrilateral elements. Convergence of the approximate basis $\{ \hat{u}^{\varepsilon}_i \}_{i=1}^3$ towards the reference solution $\{ \hat{u}^*_i \}_{i=1}^3$ has been illustrated in Figure \ref{fig:squareconv3}. In this case the error behaves like $N_{\mc{A}_{\varepsilon}(\eta)}^{-1.5}$ with respect to the number of collocation points.

\begin{figure}[htb]
\begin{center}
\subfloat[{As a function of $\varepsilon$.}]{\includegraphics[width=0.48\textwidth]{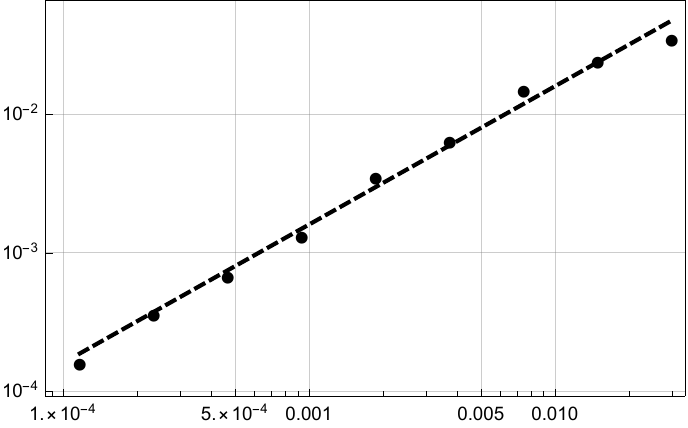}}\quad
\subfloat[{As a function of $N_{\mc{A}_{\varepsilon}(\eta)}$.}]{\includegraphics[width=0.48\textwidth]{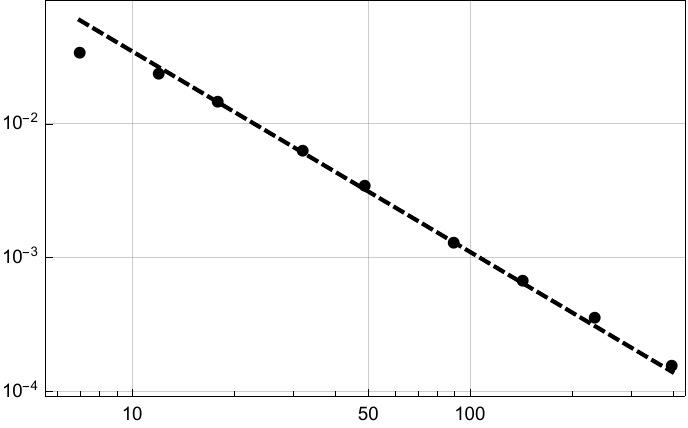}}
\caption{Model problem in the unit square when $\varsigma = 3$: Convergence of the approximate solution $\{ \hat{u}^{\varepsilon}_i \}_{i=1}^3$ to the reference solution $\{ \hat{u}^*_i \}_{i=1}^3$. The points represent values of the error measure $\theta_{\epsilon}$ on a log-log scale. Dashed lines represent algebraic rates $\varepsilon^{1.0}$ and $N_{\mc{A}_{\varepsilon}(\eta)}^{-1.5}$.}
\label{fig:squareconv3}
\end{center}
\end{figure}

\subsection{General uncertainty model in a dumbbell shaped domain}

In this section we consider the stochastic diffusion problem \eqref{eq:modelproblem} in a dumbbell shaped domain. The geometry of the domain $D$ is illustrated in Figure \ref{fig:dumbbelldomain}. Again we let $a_0 := 1 + C_D$, where $C_D$ is the Poincar\'e constant for $D$. For $m \in \N$ and $\varsigma > 1$ we set
\[
a_m(x) := (m + 1)^{-\varsigma} \sin\left( \frac{2 \pi k_m x_1}{2 + w_D} \right) \sin( \pi l_m x_2 ), \quad x = (x_1, x_2) \in D,
\]
where $(k_m, l_m)$ denotes the $m$:th element of $\N^2$ with respect to increasing graded lexicographic order. As previously, the assumptions \eqref{eq:ellipticity} and \eqref{eq:decay} hold with $\alpha_0 = 1$ and $\kappa_m = (m+1)^{-\varsigma}$, whereas for $\varsigma \ge 2$ we have $\norm{\kappa}{\ell^1(\N)} < 1$. In the following examples we again use the values $\varsigma = 3$ and $\varsigma = 6$ and define the multi-index sets $\mc{A}_{\varepsilon}(\eta) \subset (\N_0^{\infty})_c$ according to the equations \eqref{eq:miset} -- \eqref{eq:rho}.

\begin{figure}[htb]
\begin{center}
\includegraphics[width=0.48\textwidth]{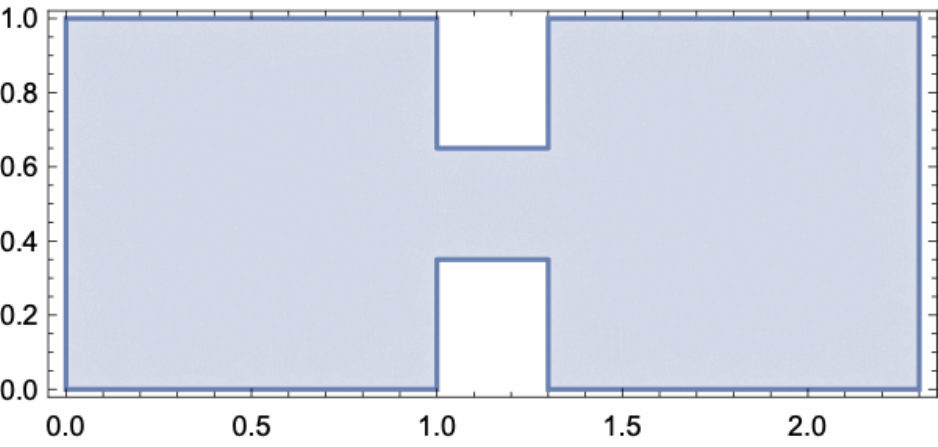}
\caption{The dumbbell domain: Two unit squares connected by a small middle part of size $h_D = w_D = 3/10$ (height and width).}
\label{fig:dumbbelldomain}
\end{center}
\end{figure}

In this example we focus on the subspace $U_J$ for $J = \{ 3,4,5,6 \}$. Figure \ref{fig:dumbbellcrossing} illustrates the corresponding eigenvalues as a function of the first parameter $y_1 \in [-1,1]$ when the rest are held constant. Now it seems that all four eigenvalues are tightly clustered and we observe multiple crossing points for the eigenvalues. An example of the associated eigenfunctions at $y = 0$ has been shown in Figure \ref{fig:dumbbellvecs}. We may again plug the calculated eigenvalues and an estimate of $\norm{\kappa}{\ell^1(\N)}$ into \eqref{eq:separationcriterion}, and see that Corollary \ref{cor:separated} holds for both $S = 2$ and for $S = 6$. With this we conclude that the subspace $U_{\{ 3,4,5,6 \}}$ is isolated.

\begin{figure}[htb]
\begin{center}
\subfloat[{Eigenvalues $\{ \mu_i \}_{i=1}^8$.}]{\includegraphics[width=0.48\textwidth]{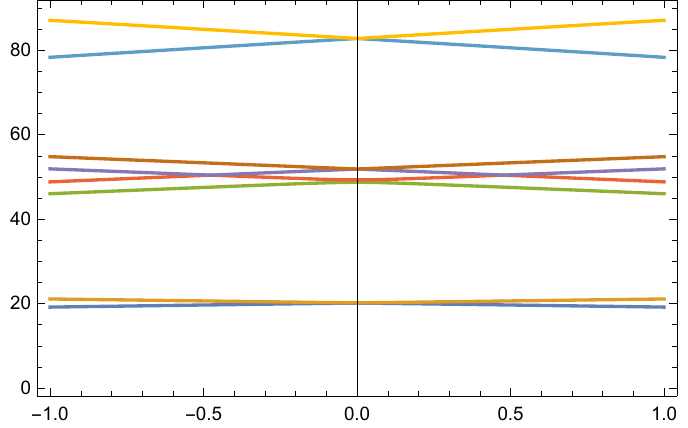}}\quad
\subfloat[{Eigenvalues $\{ \mu_i \}_{i=3}^6$.}]{\includegraphics[width=0.48\textwidth]{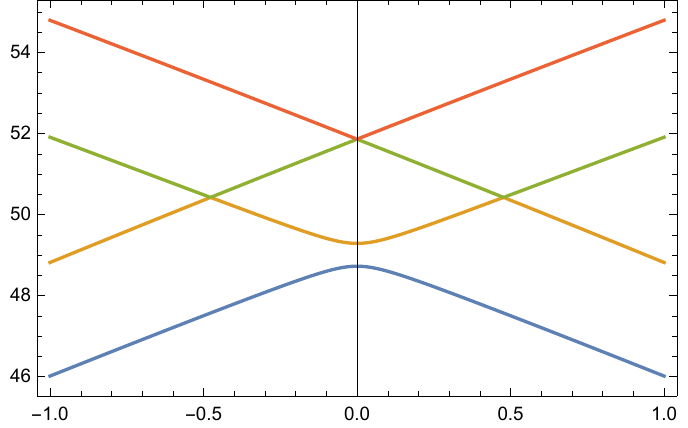}}
\caption{Model problem in the dumbbell domain with $\varsigma = 3$: First few eigenvalues as a function of $y_1 \in [-1,1]$ when $y_2 = y_3 = \ldots = 0$.}
\label{fig:dumbbellcrossing}
\end{center}
\end{figure}

\begin{figure}[htb]
\begin{center}
\subfloat[{Eigenfunction $u_3$.}]{\includegraphics[width=0.46\textwidth]{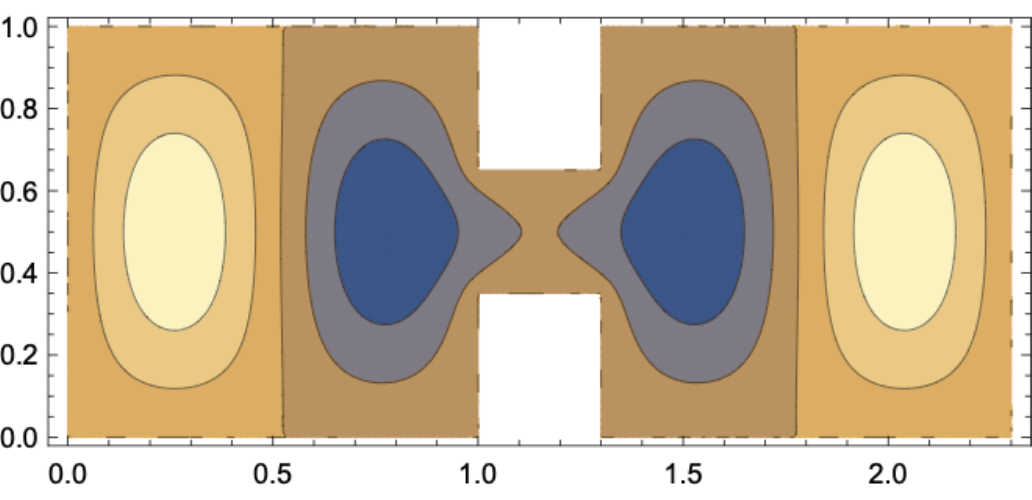}}\qquad
\subfloat[{Eigenfunction $u_4$.}]{\includegraphics[width=0.46\textwidth]{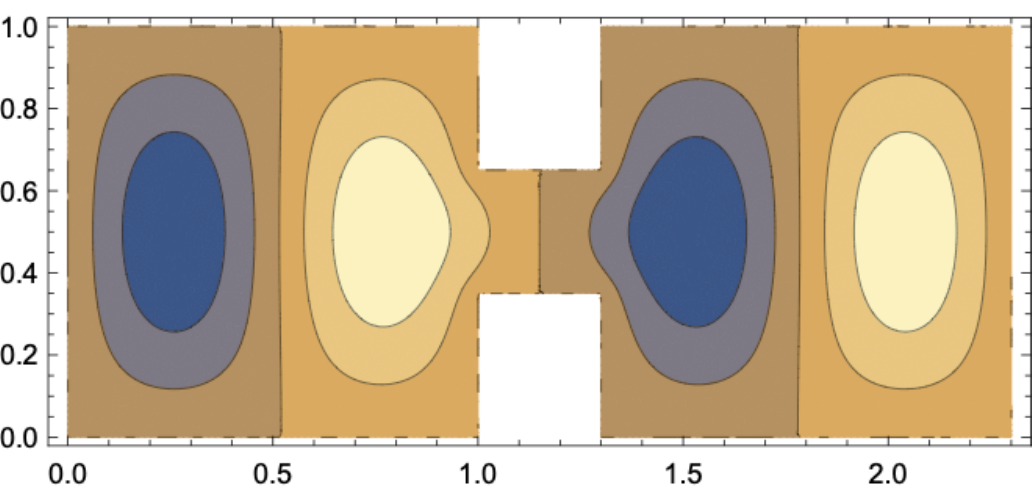}}\qquad
\subfloat[{Eigenfunction $u_5$.}]{\includegraphics[width=0.46\textwidth]{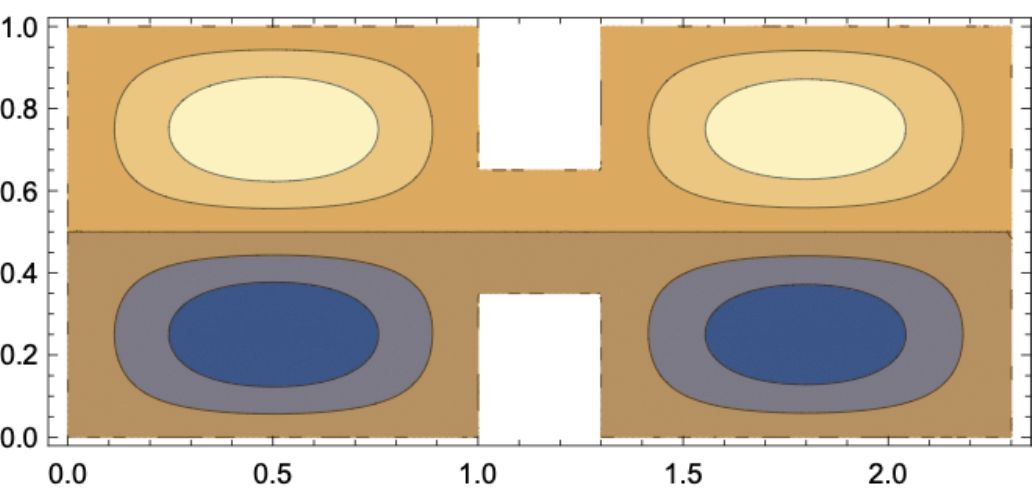}}\qquad
\subfloat[{Eigenfunction $u_6$.}]{\includegraphics[width=0.46\textwidth]{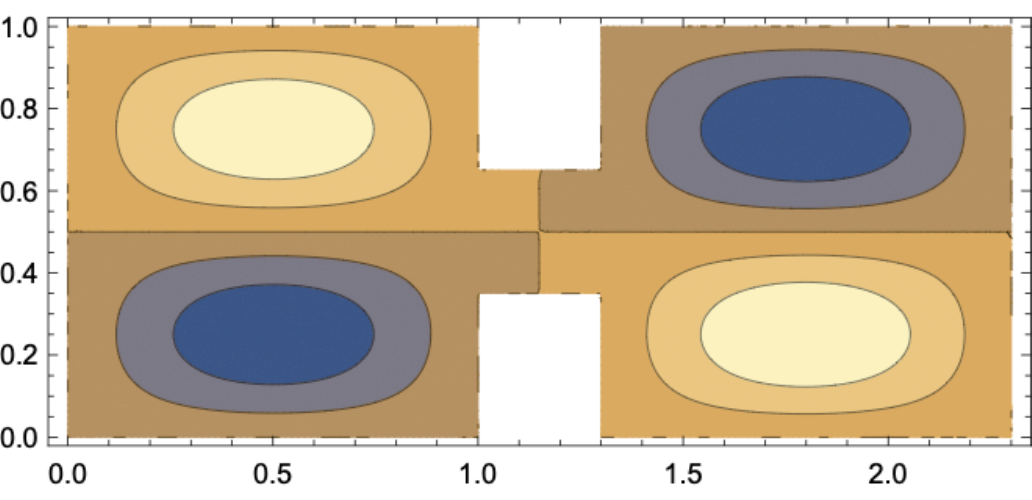}}
\caption{Model problem in the dumbbell domain: Eigenfunctions 3 to 6 at $y = 0$.}
\label{fig:dumbbellvecs}
\end{center}
\end{figure}

We start our convergence analysis from the case $\varsigma = 6$. We compute a reference solution $\{ \hat{u}^*_i \}_{i=3}^6$ using a mesh of 220756 second order triangular elements. In computing this reference solution we set $\varepsilon > 0$ so that the number of multi-indices is $\# \mc{A}_{\varepsilon}(\eta) = 28$ and the greatest active dimension is $M_{\mc{A}} = 16$. This results in $N_{\mc{A}_{\varepsilon}(\eta)} = 77$ collocation points. We then compute a series of solutions $\{ \hat{u}^{\varepsilon}_i \}_{i=3}^6$ using the same deterministic mesh. Convergence of these approximate basis vectors $\{ \hat{u}^{\varepsilon}_i \}_{i=3}^6$ towards the reference solution $\{ \hat{u}^*_i \}_{i=3}^6$ with respect to the error measure
\[
\theta_{\epsilon} := \left( \sum_{i=3}^6 \norm{\hat{u}^{\varepsilon}_i - \hat{u}^*_i}{L^2_{\nu}(\Gamma) \otimes H^1(D)}^2 \right)^{1/2}
\]
has been illustrated in Figure \ref{fig:dumbbellconv6}. The error behaves like $N_{\mc{A}_{\varepsilon}(\eta)}^{-3.0}$ with respect to the number of collocation points.

\begin{figure}[htb]
\begin{center}
\subfloat[{As a function of $\varepsilon$.}]{\includegraphics[width=0.48\textwidth]{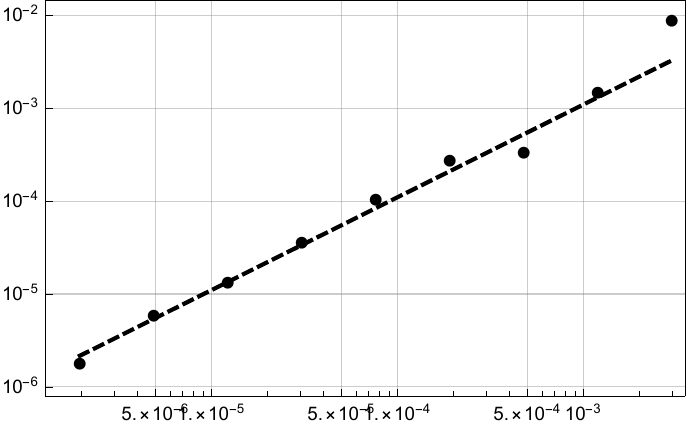}}\quad
\subfloat[{As a function of $N_{\mc{A}_{\varepsilon}(\eta)}$.}]{\includegraphics[width=0.48\textwidth]{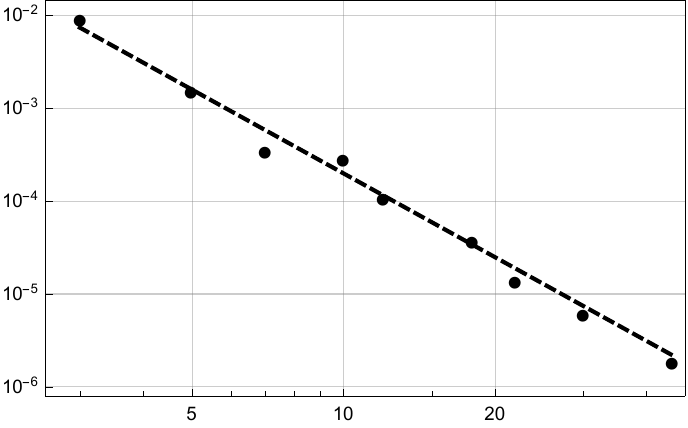}}
\caption{Model problem in the dumbbell domain when $\varsigma = 6$: Convergence of the approximate solution $\{ \hat{u}^{\varepsilon}_i \}_{i=1}^3$ to the reference solution $\{ \hat{u}^*_i \}_{i=1}^3$. The points represent values of the error measure $\theta_{\epsilon}$ on a log-log scale. Dashed lines represent algebraic rates $\varepsilon^{1.0}$ and $N_{\mc{A}_{\varepsilon}(\eta)}^{-3.0}$.}
\label{fig:dumbbellconv6}
\end{center}
\end{figure}

Finally we repeat the previous exercise for $\varsigma = 3$. Again, Theorem \ref{thm:convergencepoints} does not hold for this value of $\varsigma$ but we still observe convergence numerically. In this case we use a mesh of 10074 second order triangular elements. We compute a reference solution $\{ \hat{u}^*_i \}_{i=3}^6$ with $\varepsilon > 0$ such that the number of multi-indices is $\# \mc{A}_{\varepsilon}(\eta) = 302$ and the greatest active dimension is $M_{\mc{A}} = 129$, which gives us $N_{\mc{A}_{\varepsilon}(\eta)} = 1053$ collocation points. We then compute a series of solutions $\{ \hat{u}^{\varepsilon}_i \}_{i=3}^6$ using different values of $\varepsilon > 0$ and the same deterministic mesh. Convergence of the approximate basis $\{ \hat{u}^{\varepsilon}_i \}_{i=3}^6$ towards the reference solution $\{ \hat{u}^*_i \}_{i=3}^6$ has been illustrated in Figure \ref{fig:dumbbellconv3}. Again the error behaves like $N_{\mc{A}_{\varepsilon}(\eta)}^{-1.5}$ with respect to the number of collocation points.

\begin{figure}[htb]
\begin{center}
\subfloat[{As a function of $\varepsilon$.}]{\includegraphics[width=0.48\textwidth]{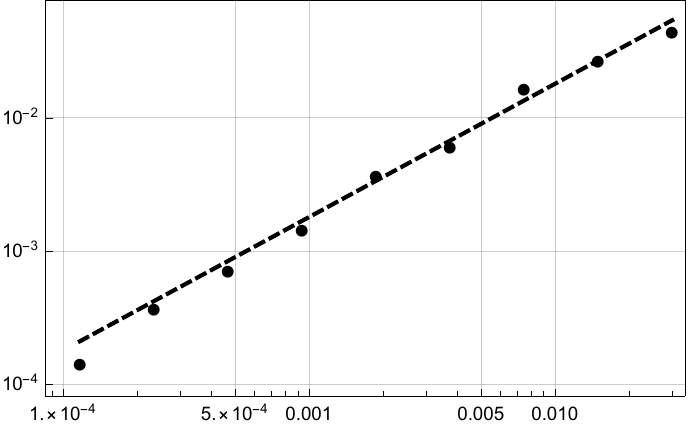}}\quad
\subfloat[{As a function of $N_{\mc{A}_{\varepsilon}(\eta)}$.}]{\includegraphics[width=0.48\textwidth]{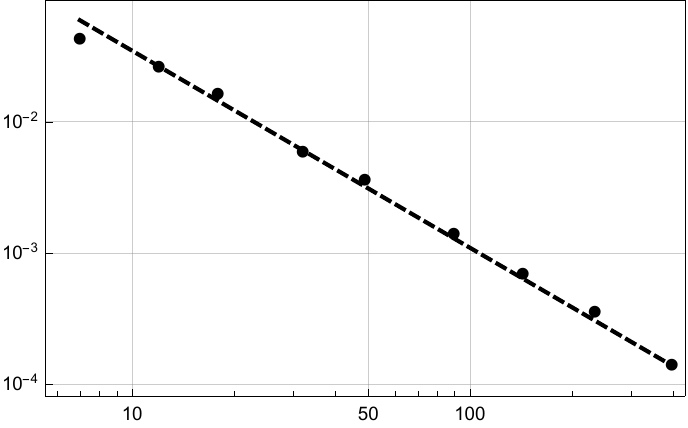}}
\caption{Model problem in the dumbbell domain when $\varsigma = 3$: Convergence of the approximate solution $\{ \hat{u}^{\varepsilon}_i \}_{i=1}^3$ to the reference solution $\{ \hat{u}^*_i \}_{i=1}^3$. The points represent values of the error measure $\theta_{\epsilon}$ on a log-log scale. Dashed lines represent algebraic rates $\varepsilon^{1.0}$ and $N_{\mc{A}_{\varepsilon}(\eta)}^{-1.5}$.}
\label{fig:dumbbellconv3}
\end{center}
\end{figure}

\section{Conclusions and future prospects}
\label{sec:conclusions}

We have studied the eigenvalue problem of an operator that depends affinely on a countable number of input parameters. We have shown that if a set of eigenvalues is strictly separated from the rest of the spectrum, then the subspace spanned by the corresponding eigenvectors exhibits analytic dependence on the input parameters. We have then defined a set of canonical basis vectors that span this subspace and are smooth also in the vicinity of eigenvalue crossings. Hence, stochastic collocation methods, with known rates of convergence, may be applied to compute these canonical basis vectors.

In our numerical examples we have applied a sparse multi-index stochastic collocation algorithm to computing subspaces of a stochastic diffusion operator written in its Karhunen-Lo\`eve expansion. Our examples show that optimal rates of convergence hold even in the presence of eigenvalue crossings. In fact, in our examples we observe fast rates of convergence even if the terms in the Karhunen-Lo\`eve series decay too slowly for the current theory to hold. The validity of our collocated solution has been verified by comparing the results of our subspace algorithm to the results of a simple eigenvalue algorithm applied to the same problem in a dimensionally reduced form. We note, that a computationally more efficient solver for the problem at hand could be obtained by a sparse composition of the stochastic and spatial approximation operators, see e.g. \cite{bieriandreevschwab09} and \cite{andreevschwab12}.

In the current paper we have introduced an algorithm for computing a basis for the eigenspace of interest with the drawback that the individual eigenvalues and eigenvectors are lost in this process. In some cases, see for instance \cite{hakulalaaksonen19a}, we could try to regain the eigenvalues by tracking smooth branches of the eigenmodes within the parameter space. However, in the general case with more than one parameter, smooth branches might not always exist: See the example by Rellich given in \cite{reedsimon78}, page 60. In order to overcome this problem, we would need to consider non-smooth solution methods, i.e., ones that does not rely on the analyticity of the solution. This topic is left for future research.

\bibliography{bibfile}

\end{document}